\documentclass[reqno,10.8pt]{amsart}
\usepackage{amsmath}
\usepackage{amsfonts}
\usepackage{amssymb}
\usepackage{latexsym}

\newtheorem{theorem}{Theorem}
\newtheorem{proposition}[theorem]{Proposition}
\newtheorem{remark}[theorem]{Remark}

\usepackage{bbm,amsmath,amsthm,epsfig,latexsym,marvosym}
\usepackage{amsfonts}
\usepackage{bigints}

\def\R{\mathbb R}

\def\e{\varepsilon}

\def\Om{\Omega}

{\left\lbrace\begin{array}{@{}l@{}}}%
{\end{array}\right.}

\def\pa{\partial}

\def\d{\, \mathrm{d}}

\DeclareMathOperator*{\hd}{hd}

\title{A note on the stability of the Cheeger constant of $N$-gons}

\author{M. Caroccia}
\address{Dipartmento di Matematica, Universit\`a di Pisa, Largo Bruno Pontecorvo 5, 56127 Pisa, Italy}
\email{caroccia.marco@gmail.com}

\author{R. Neumayer}
\address{Department of Mathematics, University of Texas at Austin, Austin, TX, USA}
\email{rneumayer@math.utexas.edu}

\begin{document}
\maketitle
\begin{abstract}
The regular $N$-gon provides the minimal Cheeger constant in the class of all $N$-gons with fixed volume. This result is due to a work of Bucur and Fragal\`a in 2014. In this note, we address the stability of their result in terms of the $L^1$ distance between sets. Furthermore, we provide a stability inequality in terms of the Hausdorff distance between the boundaries of sets in the class of polygons having uniformly bounded diameter. Finally, we show that our results are sharp, both in the exponent of decay and in the notion of distance between sets.
\end{abstract}
\text{}\\
Given a Borel set $\Om \subset \R^n$ with finite measure, the \textit{Cheeger constant} of $\Om$ is defined by 
$$h(\Om) = \inf \left\{ \frac{P(E)}{|E|}\  \Big{|} \   E \text{ measureable, }E\subseteq \Om\right\}.$$
\noindent One can show that the infimum above is always attained, and a set $E \subseteq \Om$ such that $\frac{P(E)}{|E|} = h(\Om)$ is called a \textit{Cheeger set} of $\Om$.

% One easily sees that for any set $\Om$,
 %the boundary of a Cheeger set $C$ of $\Om$ touches the boundary of $\Om$, otherwise, by enlarging $C$, $h(\Om)$ is decreased.
 % It doesn't seems really useful to understand the problem.

\noindent The Cheeger constant is related to the first Dirichlet eigenvalue of the $p$-Laplacian of a set $\Om\subset \R^n$:
\begin{equation}\label{p-lap eigen}
\lambda_p(\Om):=\inf\left\{\int_{\Om} |\nabla u|^p \d x \ \Big{|} \ u\in W^{1,p}_0(\Om), \  \|u\|_{L^p(\Om)}=1\right\}.
\end{equation}
Indeed, for every set $\Om$ and every $p>1$,
\begin{equation}\label{lb p-eig}
\lambda_p(\Om)\geq \left(\frac{h(\Om)}{p}\right)^p \ \ \text{and} \quad \lim_{p\rightarrow 1} \lambda_p(\Om)=h(\Om).
\end{equation}
See, for example, \cite{KN08} for more details about the relation between the Cheeger constant and the first Dirchlet eigenvalue of the p-Laplacian  or \cite{Bu10} for more details about spectral problems. It is a well known fact that for every $p>1$, the ball provides the minimal value for $\lambda_p$ among all sets with fixed volume. The \textit{Cheeger inequality} expresses this fact for the limit case $p=1$:
\begin{equation}\label{chee in}
|\Om|^{\frac{1}{n}}h(\Om)\geq  h(B), \ \ \ \text{whenever $\Om$ is a Borel set},
\end{equation}
where $B$ denotes the unit-area ball of $\R^n$. A celebrated conjecture due to P\' olya and Szeg\H{o}, states that if one considers the minimization problem for $\lambda_p$ among the class of $N$-gons with fixed area then the solution is given by the regular $N$-gon. In \cite{PS51} the authors show the validity of the conjecture for $p=2$ in the case $N=3$ and $N=4$. \\
\text{}\\
Recent progress has been made in \cite{BF14}, where the authors prove the validity of the conjecture for the limit case $p=1$ (i.e. the Cheeger constant). In particular, they prove that for every fixed $N$, the regular $N$-gon minimizes the Cheeger constant among all $N$-gons with fixed volume: 
 \begin{equation}\label{BF}\sqrt{|\Om|}h(\Om) \geq h(\Om_0) \quad \text{ for } \Om \in P_N,
\end{equation}
where here and in the sequel $P_N $ is the set of all $N$-gons in $\R^2$ and $\Om_0$ denotes the unit-area regular $N$-gon.
Our goal in this short note is to address the stability of \eqref{BF} in the class of $N$-gons. Namely, our theorem is the following.
\begin{theorem}\label{gblstability} 
For each $N \geq 3$, there exists $C$ and $\eta>0$ depending only on $N$ such that if $\Omega \in P_N$ with $|\Omega| = 1$ and $h(\Omega) - h(\Om_0) \leq \eta,$ then there exists a rigid motion $\rho$ of $\mathbb{R}^2$ such that 
\begin{equation}\label{hdineq}
\hd(\partial \Omega, \partial \rho \Omega_0)^2 \leq C ( h(\Omega) - h(\Om_0)).
\end{equation}
As a consequence, for every $N \geq 3$, there exists some constant $C_1$ depending only on $N$ such that for any $\Omega \in P_N$ with $|\Omega|=1$, there exists a rigid motion $\rho$ of $\mathbb{R}^2$ such that
\begin{equation}\label{L1}
|\Om \Delta \rho \Om_0|^2 \leq C_1 \left( h(\Om) - h(\Omega_0)\right).
\end{equation}
Moreover, for every $M>0$, there exists a constant $C_2$ depending on $N$ and $M$ such that for any $\Om \in P_N$ with $\text{diam}(\Om) <M$ and $|\Om| = 1$, the following holds:
\begin{equation}\label{hdM}
\hd(\partial \Om, \partial\rho\Om_0)^2  \leq C_2 \left( h(\Om) - h(\Omega_0)\right),
\end{equation}
where $\rho$ is a suitable rigid motion of $\mathbb{R}^2$.
\end{theorem}
Here $\hd(\cdot,\cdot)$ denotes the Hausdorff distance between sets. As we will show below, in general we cannot expect to have stability in the form of \eqref{hdineq} when $h(\Om)$ is far from being optimal, as the Cheeger constant does not detect the presence of small tentacles (see Figure \ref{counterex}), even in the class of $N$-gons. \\
\text{}\\
Before moving on to the proof of Theorem \ref{gblstability}, for sake of completeness we briefly retrieve results, definitions, and terminology that we use from \cite{KL06} and \cite{BF14}.
For a convex subset $\Om\subseteq \R^2$ there exists a unique Cheeger set $C$ with the the following characterization: $C = \Om^{(i)} \oplus B_R,$ where
$R = \frac{1}{h(\Om)}$ and 
$$\Om^{(i)} = \{ x \in \Om \ | \ \text{dist}(x,\partial \Om) >R\}$$
is called the \textit{inner parallel set of $\Om$.}  
Here, given two sets $E$ and $F$, $E\oplus F = \{ x+y\ |\ x\in E, y\in F\}$. An $N$-gon $\Om$ is called \textit{Cheeger regular} if the boundary of its Cheeger set touches all $N$ sides of $\Om$.
In \cite{KL06}, it is shown that if a set $\Om \in P_N$ is convex and Cheeger regular, then
\begin{equation}\label{CRformula}
h(\Om) = \frac{ P(\Om) + \sqrt{ P(\Om)^2 - 4\tau(\Om)|\Om|}}{2|\Om|},
\end{equation}
where 
$$\tau(\Om) = \sum_{i=1}^N\left[ \tan \left( \frac{\pi - \gamma_i}{2}\right) - \left(\frac{\pi - \gamma_i}{2}\right) \right]$$
and where $\{\gamma_i\}_{i=1}^{N}$ are the inner angles of the polygon.
In the general case, the stability of the Cheeger inequality was obtained in \cite{FMP10} as a consequence of the quantitative isoperimetric inequality for sets of finite perimeter (\cite{FMP09}, \cite{FuMP06}, \cite{CL10}). As in the general case, Theorem \ref{gblstability} is a consequence of the stability of the polygonal isoperimetric inequality, proved in \cite{IN14}. More precisely, we make use of the following proposition proved in \cite{CM14} as a consequence of \cite{IN14}.
\begin{proposition}\label{isoineSharp}
For every $N\geq 3$, there exists a positive constant $C$ depending only on $N$ with the following property: for every convex unit-area $\Om\in P_N$, there exists a rigid motion $\rho$ of $\R^2$ such that
\begin{equation}
    \label{quantitative indrei inq}
    \hd(\pa \Om,\pa \rho\Om_0 )^2 \leq  C\,\left(P(\Om)^2-P(\Om_0)^2\right).
\end{equation}
\end{proposition}
Finally, we recall the isoperimetric inequality for convex polygons, which states that, for every convex polygon $\Om$,
\begin{equation}\label{polyisop}\frac{P(\Omega)^2}{4|\Omega|} \geq \tau(\Omega) + \pi,
\end{equation}
with equality if and only if $\Omega $ is a circumscribed polygon (namely a polygon for which the largest inscribed circle touches all of his sides). 

In their paper, Bucur and Fragal\`a remark that in the case of simple convex polygons,  the proof of their theorem becomes straightforward by exploiting the characterization of the Cheeger set for a convex bounded Borel set $\Om$ contained in \cite{KL06}. We are able to reduce to this case with a compactness argument, then, by tracing through their argument and making use of Proposition~\ref{isoineSharp}, we can easily achieve the proof. 

\begin{proof}[Proof of Theorem \ref{gblstability}]
Notice that $|\Om \Delta \rho \Om_0| \leq 2 $ for all $\Omega \subset \mathbb{R}^2$ with $|\Om|=1$ and for every rigid motion $\rho$ of $\R^2$, so inequality \eqref{L1} follows immediately from \eqref{hdineq} by choosing the constant to be sufficiently large. Furthermore, \eqref{hdM} is a consequence of \eqref{hdineq}. Indeed, 
$$\hd( \partial E , \partial F) \leq \text{diam}(E) + \text{diam}(F),$$
 for two generic closed sets $E$ and $F$ in $\mathbb{R}^2$ such that
\begin{equation}\label{diameter}
\partial E \cap \partial F \neq \emptyset.
\end{equation}
 Up to a translation, \eqref{diameter} holds for $\Omega$ and $\Om_0$; this and the boundedness of the diameter imply that 
$$\hd(\partial \Om, \partial \Om_0) \leq M + \text{diam}(\Omega_0),$$
so by choosing the constant large enough depending on $M$, we obtain \eqref{hdM}. Therefore let us focus on the proof of \eqref{hdineq}. We divide the proof into two steps. \\
\text{}\\
\textit{Step 1: Qualitative stability.} We prove that for every fixed $\e>0$ there exists $\eta_0$ such that if $\Om$ is a unit-area $N$-gon with $h(\Om)-h(\Om_0)<\eta_0$, then for some rigid motion $\rho$ of $\R^{2}$, we have $\hd(\pa \Om,\pa\rho \Om_0)<\e$. Indeed, consider a sequence of unit-area $N$-gons $\Om_k$ such that $h(\Omega_k) \to h(\Om_0)$.  
Following the compactness argument in Proposition $9$ in \cite{BF14}, we obtain a subsequence of $\Omega_k$ converging in $L^1_{\text{loc}}$ to a limit $\Pi$ with $|\Pi| \leq 1$. We let $\Pi_0 = \sqrt{  |\Pi|^{-1}} \Pi.$ Then, as shown in \cite{BF14},
$$h(\Pi_0)  = \sqrt{  |\Pi|} h(\Pi) \leq \underset{k \to \infty } \liminf \sqrt{  |\Pi|} h(\Omega_k) =  h(\Omega_0).$$
On the other hand, by minimality, 
$$h(\Omega_0) \leq h(\Pi_0).$$
Therefore, $\Pi_0= \Omega_0$ up to a rigid motion, and so $\Pi$ is a regular $N$-gon.  Consider a ball $B_R$ such that $\Pi \subset\subset B_R$ and suppose that there is some subsequence of $\Omega_k$ such that each $\Omega_k$ has a vertex not contained in $B_R$. Then we conclude that $\Pi$ must have at most $N-1$ vertices, contradicting the fact that $\Pi$ is a regular $N$-gon. 
Thus $\Om_k \subset B_R$ for $k$ large, and, in particular, $\Omega_k \to \Pi$ in $L^1$. Therefore $|\Pi| = 1$ and thus $\Pi=\rho\Om_0$ for some rigid motion $\rho$. 
If a bounded sequence of $N$-gons converges in $L^1$ to an $N$-gon, then the boundaries converge in the Hausdorff distance, so 
$$\hd (\partial \Omega_k, \partial \rho \Omega_0) \to 0$$ 
and we achieve the proof of Step 1.\\
\text{}\\
\textit{Step 2: Quantitative stability.} Fix $\e>0$, let $\eta_0$ be the constant given from the Step 1 and let $\Om$ be a unit-area $N$-gon such that $h(\Om)-h(\Om_0)<\eta$ for some $\eta<\eta_0$ to be fixed. Thanks to Step 1 and because $\Omega_0$ is convex, up to choosing an $\e$ small enough,  $\Omega$ must be as well. Moreover, since $\Omega_0$ is Cheeger regular, $\Omega$ will also be Cheeger regular up to further decreasing $\e$ and with a suitable choice of $\eta$. Indeed, we note that a sufficient condition for an $N$-gon to be Cheeger regular is for its inner parallel set to also be an $N$-gon.
The inner parallel set $\Omega^{(i)}$ have boundary given by
$$\partial \Omega^{(i)} = \left\{ x \in \Omega \ \Big{|} \ \text{d}(x, \partial \Omega) = \frac{1}{h(\Om)} \right\}.$$
Therefore, since
$\hd (\partial \Omega, \partial \Omega_0)\leq \e$ and $h(\Om)-h(\Om_0)<\eta$ up to a rigid motion of $\R^{2}$:
$$\hd\left( \pa \Omega^{(i)}, \pa \Omega_0^{(i)}\right) \leq \omega(\e+\eta) ,$$
for a continuous function $\omega$ such that $\omega(0_+) = 0$.
In particular, since the inner parallel set of $\Om_0$ is an $N$-gon, by choosing $\e$ and $\eta<\eta_0(\e)$ small enough, the previous arguments and Step 1 prove that $\Omega$ is a convex, Cheeger regular $N$-gon. 
We are thus in the position to exploit \eqref{CRformula}, the characterization of the Cheeger constant for convex, Cheeger regular sets. We obtain the following:
\begin{equation}\label{bound1}\begin{split}
2(h(\Om)- h(\Om_0)) & = 
P(\Omega) + \sqrt{ P(\Omega)^2 - 4 \tau(\Omega)}-
P(\Omega_0) + \sqrt{ P(\Omega_0)^2 - 4\tau(\Omega_0)}\\
& = 
P(\Omega)- P(\Omega_0)
+ \sqrt{ P(\Omega)^2 - 4 \tau(\Omega)}- \sqrt{ 4\pi} \\
& \geq P(\Omega)- P(\Omega_0),
\end{split}
\end{equation}
where the final two lines both follow from \eqref{polyisop}.
Moreover, we have
\begin{equation}\label{bound2}
P(\Omega)- P(\Omega_0) \geq c(P(\Omega)^2 - P(\Omega_0)^2)
\end{equation}
for $\e$ small enough, because $\hd(\partial \Omega, \partial \Omega_0) \leq  \e$ and thus $P(\Omega) <2P(\Omega_0)$. Finally, Proposition~\ref{isoineSharp} combined with \eqref{bound1} and \eqref{bound2} yields
$$ h(\Om) - h(\Om_0) \geq c\mbox{ } \hd (\partial \Omega, \partial \rho\Omega_0)^2$$
for $c >0$ where $c$ depends only on $N$.
\end{proof}

\begin{figure}[t]
\begin{center}
\includegraphics[scale=0.65]{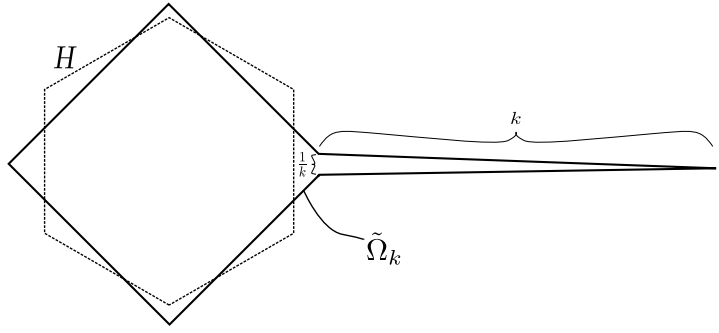}\caption{\small{A sequence of hexagons with Cheeger constants uniformly bounded and diameters unbounded. In this case, the Cheeger deficit cannot control the Hausdorff distance from a unit-area regular hexagon.}}\label{counterex}
\end{center}
\end{figure}

\begin{remark}{\rm 
We point out that \eqref{hdineq} fails to hold when $h(\Om) - h(\Om_0)$ is large, as the Cheeger constant does not detect behavior of the set away from where the Cheeger set lies. Consider $N=6$ for the sake of simplicity, and consider the following
 construction as in Figure ˜\ref{counterex}. 
First obtain $\tilde{\Om}_{k}$ by cutting away a triangle of base $\frac{1}{k}$ from a corner of a unit-area square $S$ and replacing it with another triangle having the same base but with height $k$. Set $\Omega_k: = \frac{\tilde{\Om}_k}{|\tilde{\Om}_k|}$ and note that their Cheeger constants are uniformly bounded.
Indeed, if $C_0$ is the Cheeger set of $S$, then by construction, $C_0 \subset \tilde{\Om}_k$ for all $k$, and thus 
$h(\tilde{\Om}_k ) \leq h(S), $
so
$$h(\Om_k ) = \sqrt{|\tilde{\Om}_k|} h(\tilde{\Om}_k) \leq  \sqrt{|\tilde{\Om}_k|} h(S) \leq\sqrt{\frac{3}{2}} h(S).$$
 On the other hand, clearly
$$\inf\{\hd(\partial \Omega_k, \partial\rho \Omega_0 )\ | \ \rho \text{ rigid motion of }\R^2\}\to \infty.$$
}
\end{remark}

\begin{figure}[h]
\begin{center}
\includegraphics[scale=0.55]{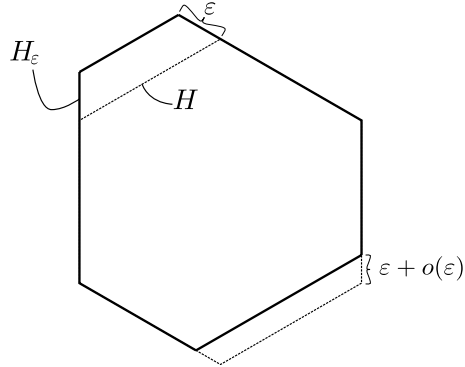}\caption{\small{We choose a small perturbation of the regular hexagon $H$ in an area-fixing way. Some easy calculations show that $h(H_{\e})-h(H)\approx c_0 A(\e)^2=c_1 |H_{\e} \Delta H|^2\approx c_2\hd(\pa H_{\e},\pa H)^2$, where $A(\e)$ denotes the areas bounded by the dotted lines. Thus yields the sharpness of \eqref{L1} and \eqref{hdM}. }}\label{smallperthex}
\end{center}
\end{figure}

\begin{remark}{\rm We also underline that the exponent $2$ in \eqref{L1} and \eqref{hdM} cannot be improved, due to the sharpness of the quantitative isoperimetric inequality for polygons. Indeed, suppose that, for example, inequality \eqref{L1} were to hold as
\begin{equation}
f(|\Om \Delta\rho \Om_0|)  \leq C_1 \left( h(\Om) - h(\Omega_0)\right),
\end{equation}
for some $f(x)$. Then by testing the inequality on a small area-fixing perturbations of the regular $N$-gon (as in figure \ref{smallperthex}) we immediately find that $f(|\Om_{\e} \Delta \rho \Om_0|)\leq c |\Om_{\e} \Delta \rho \Om_0|^2$.
}

\end{remark}

{\bf \noindent Acknowledgments.} We thank Berardo Ruffini for bringing the problem to our attention and Francesco Maggi for a careful reading of this note and for his useful comments. The work of MC was partially supported by the project 2010A2TFX2 "Calcolo delle Variazioni" funded by\textit{ the Italian Ministry of Research and University} and NSF Grant DMS-1265910. He also acknowledges the hospitality at the UT Austin during the 2014 fall semester where the present work has been done. The work of RN was supported by the NSF Graduate Research Fellowship under Grant DGE-1110007.


\begin{thebibliography}{}

\bibitem[BF14]{BF14} Bucur, Fragal\`a, \emph{A Faber-Krahn inequality for the Cheeger Constant of $N$-gons}, Journal of Geometric Analysis (2014).
\\
\bibitem[Bu10]{Bu10} Buttazzo, \emph{Spectral optimization problem}, Revista Matem‡tica Complutense, 24(2), p.277-322 (2011).
\\
\bibitem[CM14]{CM14} Caroccia, Maggi, \emph{A sharp quantitative version of Hales' isoperimetric honeycomb theorem} Arxiv 1410.6128 (2014).
\\
\bibitem[CL10]{CL10} Cicalese, Leonardi, \emph{A selection principle for the sharp quantitative isoperimetric inequality}, Archive for Rational Mechanics and Analysis (2010).
\\
 \bibitem[FMP09]{FMP09} Figalli, Maggi, Pratelli, \emph{Note on cheeger sets}, Proc. Amer. Math. Soc. 137 (6), p.2057-2062 (2009). 
\\
\bibitem[FMP10]{FMP10} Figalli, Maggi, Pratelli, \emph{A mass transportation approach to quantitative isoperimetric inequailities}, Invent. Math. 182, no. 1, 167-211 (2010).
\\
\bibitem[FuMP06]{FuMP06} Fusco, Maggi, Pratelli, \emph{The sharp quantitative isoperimetric inequality} Annals of Mathematic (2006).
\\
\bibitem[IN14]{IN14} Indrei, Nurbekyan, \emph{On the stability of the polygonal isoperimetric inequality} Arxiv 1402.4460, (2014).
\\
\bibitem[KL06]{KL06}Kawohl, Lachand, \emph{Characterization of cheeger sets for convex subsets of the plane.} Pacific J. Math. 225(1) 103-118 (2006).
\\
\bibitem[KN08]{KN08}Kawohl, Novaga, \emph{The $p$-Laplacian eigenvalue problem as $p\rightarrow 1$ and Cheeger sets in Finsler metric} Journal of Convex Analysis (2008).
\\
\bibitem[PS51]{PS51} P\' olya, Szeg\H{o}, \emph{Isoperimetric inequalities in mathematical physics.} Annals of Mathematical Studies 27, Princeton University Press, Princeton, (1951).
\\
%\bibitem[Mag]{Mag} Maggi, Francesco., \emph{Sets of finite perimeter and geometric variational problems}, volume 135 
%of Cambridge Studies in Advanced Mathematics, Cambridge University Press, Cambridge, 2012. ISBN 978-
%1-107-02103-7. xx+454 pp. An introduction to Geometric Measure Theory.

\end{thebibliography}
\end{document}